\documentclass[dvips,preprint]{imsart}

\RequirePackage[OT1]{fontenc}
\RequirePackage{amsthm,amsmath,natbib,mathrsfs, amssymb}

\startlocaldefs \numberwithin{equation}{section}
\theoremstyle{plain}
\newtheorem{thm}{Theorem}[section]
\theoremstyle{plain}
\newtheorem{lem}{Lemma}[section]
\theoremstyle{plain}
\newtheorem{cor}{Corollary}[section]
\theoremstyle{remark}

\endlocaldefs

\begin{document}

\begin{frontmatter}
\title{Non-asymptotic Oracle Inequalities for the High-Dimensional Cox Regression via Lasso}
\runtitle{Lasso Cox Regression}
\thankstext{e3}{Supported in part by NSF Grant DMS-1007590 and NIH grant R01-AG036802.}

\begin{aug}
\author{\fnms{Shengchun} \snm{Kong}
\ead[label=e1]{kongsc@umich.edu }}
\and
\author{\fnms{Bin} \snm{Nan}
\thanksref{e3}
\ead[label=e3]{bnan@umich.edu}%
}

\address{Department of Biostatistics \\ University of Michigan \\ 1420 Washington Heights\\
Ann Arbor, MI 48109-2029\\
\printead{e1}\\
\phantom{E-mail:\ }\printead*{e3}}


\runauthor{S. Kong and B. Nan}

\affiliation{University of Michigan}

\end{aug}

\begin{abstract}
We consider the finite sample properties of the regularized
high-dimensional Cox regression via lasso. Existing literature focuses on linear models or generalized linear models with Lipschitz loss functions, where the empirical risk functions are the summations of independent and identically distributed (iid) losses. The summands in the negative log partial likelihood function for censored survival data, however, are neither iid nor Lipschitz. We first approximate the negative log partial likelihood function by a sum
of iid non-Lipschitz terms, then derive the non-asymptotic oracle inequalities for the lasso penalized Cox regression using pointwise arguments to tackle the difficulty caused by the lack of iid and Lipschitz property.
\end{abstract}

\begin{keyword}[class=AMS]
{62N02}
\end{keyword}

\begin{keyword}
\kwd{Cox regression} \kwd{finite sample} \kwd{lasso} \kwd{oracle
inequality} \kwd{variable selection}
\end{keyword}

\end{frontmatter}

\section{Introduction}
Since it was introduced by \cite{tibshirani1996}, the lasso regularized
method for  high-dimensional regression models with sparse
coefficients has received a great deal of attention in the
literature. Properties of interest for such regression models
include the finite sample oracle inequalities. Among the extensive
literature of the lasso method, \cite{b-t-wegkamp2007} and \cite{b-r-tsybakov2009} derived the oracle inequalities for prediction risk and estimation error in a general nonparametric regression
model including the high-dimensional linear regression as a special
example, and \cite{geer2008}  provided oracle inequalities for the
generalized linear models with Lipschitz loss functions, e.g.
logistic regression and classification with hinge loss.

We consider lasso regularized high-dimensional Cox regression. Let
$T$ be the survival time and $C$ the censoring time. Suppose we
observe a sequence of iid observations $(Y_i, \Delta_i, X_i)$, $i=1,
\dots, n$, where $Y_i=T_i \wedge C_i$, $\Delta_i=I_{\{T_i\leq
C_i\}}$, and $X_i$ are the covariates in $\mathscr{X}$.
Due to largely parallel material,   we follow closely the notation
in \cite{geer2008}. Let
\begin{equation*}
\mathscr{F} = \left\{ f_\theta (\cdot) =\sum_{k=1}^{m} \theta_k
\psi_k (\cdot), \theta \in \Theta \right\}.
\end{equation*}
Here $\Theta$ is a convex subset of $\bf{R}^m$, and the functions
$\psi_1, \cdots, \psi_m$ are real-valued basis functions on
$\mathscr{X}$, which are identity functions of corresponding
covariates in a standard Cox model.

Consider the following Cox model \citep{cox:72}:
$$
\lambda(t|X) = \lambda_0(t) e^{f_\theta (X)},
$$
where $\theta$ is the parameter of interest and $\lambda_0$ is the unknown baseline hazard function. The negative log partial likelihood function for $\theta$ becomes
\begin{equation} \label{PartialLikelihood}
l_n(\theta)=-\frac{1}{n} \sum_{i=1}^n \left\{f_{\theta}(X_i)- \log
\left[\frac{1}{n} \sum_{j=1}^n 1(Y_j\geq Y_i) e^{f_{\theta}
(X_j)}\right]\right\} \Delta_i.
\end{equation}
The corresponding estimator with lasso penalty is denoted by
\begin{equation*}
\hat{\theta}_n := \arg\min_{\theta \in \Theta} \{
l_n(\theta)+\lambda_n \hat{I}(\theta)\},
\end{equation*}
where
$
\hat{I}(\theta):=\sum_{k=1}^m \hat{\sigma}_k |\theta_k|
$
is the weighted $l_1$ norm of the vector $\theta \in \bf{R}^m$, with
random weights
$
\hat{\sigma}_k :=\left[{1}/{n} \sum_{i=1}^n
\psi_{k}^2(X_i)\right]^{1/2}.
$

Clearly the negative log partial likelihood is a sum of non-iid random variables. For ease of theoretical calculation, it is natural to consider the following intermediate function as a ``replacement" of the negative log partial likelihood function:
\begin{equation} \label{Intermediate}
\tilde{l}_n(\theta)=-\frac{1}{n} \sum_{i=1}^n
\left\{f_{\theta}(X_i)-\log\mu(Y_i; f_{\theta})\right\} \Delta_i,
\end{equation}
which has the desirable iid structure, but with an unknown population expectation
\begin{equation*}
\mu(t; f_{\theta})=E_{X,Y} \Big\{1(Y\geq t) e^{f_{\theta}(X)}\Big\}.
\end{equation*}
The negative log partial likelihood function (\ref{PartialLikelihood}) can then be viewed as a ``working" model for the empirical loss function (\ref{Intermediate}), and the corresponding loss function becomes
\begin{equation}\label{Loss}
\gamma_{f_\theta}=\gamma(f_{\theta}(X),Y,\Delta):=-\{f_{\theta}(X)-\log\mu(Y;
f_{\theta})\}\Delta,
\end{equation}
with expected loss
\begin{equation} \label{ExpectedLoss}
l(\theta)=-E_{Y,\Delta,X}[\{f_{\theta}(X)-\log\mu(Y;
f_{\theta})\}\Delta] = P\gamma_{f_\theta},
\end{equation}
where $P$ denotes the distribution of $(Y, \Delta, X)$. Define the target
function $\bar{f}$ by
\begin{equation*}
\bar{f}:=\arg\min_{f\in \bf{F}} P \gamma_f,
\end{equation*}
where ${\bf F} \supseteq {\mathscr F}$. For simplicity we will assume that there is a unique minimum
as in \cite{geer2008}. Uniqueness holds for the regular Cox model when ${\bf F} = {\mathscr F}$, see
 for example, \cite{andersen-gill:82}.
Define the excess risk of $f$ by
\begin{equation*}
\mathcal{E}(f):=P\gamma_f-P\gamma_{\bar{f}}.
\end{equation*}
It is desirable to show similar non-asymptotic oracle inequalities for the Cox
regression model as in, for example, \cite{geer2008} for generalized linear
models. That is, with large probability,
\begin{equation*}
\mathcal{E}(f_{\hat{\theta}_n})\leq const. \times \min_{\theta \in
\Theta} \left\{\mathcal{E}(f_{\theta})+{\cal V}_\theta\right\}.
\end{equation*}
Here ${\cal V}_\theta$ is called the ``estimation error" by \cite{geer2008}, which is typically proportional to $\lambda_n^2$ times the number of nonzero elements in $\theta$.

Note that the summands in the negative log partial likelihood function (\ref{PartialLikelihood}) are not iid, and the intermediate loss function $\gamma(\cdot, Y, \Delta)$ given in (\ref{Loss}) is not Lipschitz. Hence the conclusion of \cite{geer2008} can not be applied directly. With the Lipschitz condition in \cite{geer2008} replaced by a similar boundedness assumption for regression parameters in \cite{buhlmann2006}, we tackle the problem using pointwise arguments to obtain the oracle bounds of two types of errors: one is between empirical loss  (\ref{Intermediate}) and expected loss (\ref{ExpectedLoss}), and one is between the negative log partial likelihood (\ref{PartialLikelihood}) and empirical loss (\ref{Intermediate}).

The article is organized as follows. In Section 2, we provide
assumptions and additional notation that will be used throughout the
paper. In Section 3, following the flow of \cite{geer2008}, we first consider the case where the weights
$\sigma_k:=\left[E\psi_k^2(X)\right]^{1/2}$ are fixed, then discuss briefly the case with random weights $\hat\sigma_k$.

\section{Assumptions}
We impose five basic assumptions in this section. Assumptions A, B, and C are identical to the corresponding assumptions in \cite{geer2008}. Assumption D has a similar flavor to the assumption (A2) in \cite{buhlmann2006} for the persistency property of boosting method in high-dimensional linear regression models. Here it replaces the Lipschitz assumption in \cite{geer2008}. Assumption E is commonly used for survival models with censored data, see for example, \cite{andersen-gill:82}.

\smallskip

{\sc Assumption A.} \ $K_m:=\max_{1\leq k \leq m}\{\|\psi_k
\|_{\infty}/\sigma_k\}<\infty.$

\smallskip

{\sc Assumption B.} \ There exists an $\eta >0$ and strictly convex
increasing G, such that for all $\theta \in \Theta$ with
$\|f_{\theta}-\bar{f}\|_{\infty} \leq \eta$, one has
$
\mathcal{E}(f_{\theta}) \geq G(\|f_{\theta}-\bar{f}\|).
$

\smallskip

{\sc Assumption C.} \ There exists a function $D(\cdot)$ on the subsets of
the index set $\{1, \cdots, m\}$, such that for all $\mathscr{K}
\subset \{1, \cdots, m\}$, and for all $\theta \in \Theta$ and
$\tilde{\theta} \in \Theta$, we have
$
\sum_{k \in \mathscr{K}}\sigma_k|\theta_k-\tilde{\theta}_k|\leq
\sqrt{D( \mathscr{K})}\|f_{\theta}-f_{\tilde{\theta}}\|.
$

\smallskip

{\sc Assumption D.} \ $L_m:=\sup_{\theta\in\Theta}\sum_{k=1}^m|\theta_k|<\infty.$

\smallskip

{\sc Assumption E.} \ The observation time stops at a finite time $\tau>0$
with $\pi:=P(Y\geq \tau)>0.$

\smallskip

The convex conjugate of function $G$ given in Assumption B is denoted
by $H$ such that $uv \le G(u) + H(v)$. A typical choice of $G$ is quadratic function with
some constant $C_0$, i.e. $G(u) = u^2/C_0$, see \cite{geer2008}.

From Assumptions A, D and E, we have for any $\theta\in
\Theta$,
\begin{equation} \label{bup}
e^{|f_{\theta}(X_i)|}\leq
e^{K_mL_m\sigma_{(m)}} := U_m < \infty
\end{equation}
for all $i$, where $\sigma_{(m)}=\max_{1\leq k\leq m}\sigma_k$.

Let $I(\theta):=\sum_{k=1}^m \sigma_k |\theta_k|$ be the theoretical $l_1$ norm of $\theta$, and
$\hat{I}(\theta):=\sum_{k=1}^m \hat{\sigma}_k |\theta_k|$ be the
empirical $l_1$ norm. For any $\theta$ and $\tilde{\theta}$ in
$\Theta$, denote
\begin{equation*}
I_1(\theta|\tilde{\theta}) :=\sum_{k: \tilde{\theta}_k \neq 0}
\sigma_k |\theta_k|, ~~~~ I_2(\theta|\tilde{\theta}):=
I(\theta)-I_1(\theta|\tilde{\theta}).
\end{equation*}
Similarly we have corresponding empirical versions,
\begin{equation*}
\hat{I}_1(\theta|\tilde{\theta}) :=\sum_{k: \tilde{\theta}_k \neq 0}
\hat{\sigma}_k |\theta_k|, ~~~~ \hat{I}_2(\theta|\tilde{\theta}):=
\hat{I}(\theta)-\hat{I}_1(\theta|\tilde{\theta}).
\end{equation*}

\section{Main results}

\subsection{Non-random normalization weights in the penalty}
We show that a similar result to Theorem A.4 of \cite{geer2008} holds for the Cox model.
Suppose
that $\sigma_k=[E\psi_k^{2}]^{1/2}$ are known and consider the
estimator
\begin{equation*}
\hat{\theta}_n := \arg\min_{\theta \in \Theta} \{l_n(\theta)+\lambda_n
I(\theta)\}.
\end{equation*}

Denote the empirical probability measure  based on the sample $\{(X_i,
Y_i, \Delta_i): i=1, \dots n\}$ by $P_n$. Let $\varepsilon_1, \cdots, \varepsilon_n$ be a Rademacher
sequence, independent of the training data $(X_1, Y_1,\Delta_1),
\cdots, (X_n, Y_n,\Delta_n)$. We fix some $\theta^* \in \Theta$ and
denote $\mathscr{F}_M:=\{f_{\theta}: \theta \in \Theta,
I(\theta-\theta^*)\leq M\}$ for some $M>0$. For any $\theta$ where
$I(\theta-\theta^*)\leq M$, denote
\begin{equation*}
Z_{\theta}(M):=
\left|(P_n-P)\left[\gamma_{f_\theta}-\gamma_{f_{\theta^*}}\right]\right|= \left|\left[\tilde{l}_n(\theta)-l(\theta)\right]-\left[\tilde{l}_n(\theta^*)-l(\theta^*)\right]\right|.
\end{equation*}

Note that \cite{geer2008} has considered the supremum of the above $Z_\theta(M)$ over $\Theta$. We find that the pointwise argument is adequate for our purpose because only the lasso estimator is of interest, and that the calculation with $\sup_{f \in {\mathscr F}_M} Z_\theta(M)$ in \cite{geer2008} does not apply to the Cox model due to the lack of Lipschitz property.

\begin{lem} \label{Lemma 3.1}
Under Assumptions A, D and E, for all $\theta$ satisfying $I(\theta-\theta^*)\leq M$, we have
\begin{equation*}
EZ_{\theta}(M) \leq \bar{a}_n M,
\end{equation*}
where
\begin{equation*}
\bar{a}_n=4a_n,~~~~a_n=\sqrt{\frac{2K_m^2 \log(2m)}{n}}+\frac{K_m
\log(2m)}{n}.
\end{equation*}
\end{lem}

\begin{proof}
By the symmetrization theorem, see e.g. \cite{vaart1996} or Theorem A.2 in
\cite{geer2008}, for a class of only one function we have
\begin{eqnarray*}
EZ_{\theta}(M)&\leq&  2E\Bigg(\Bigg|\frac{1}{n}\sum_{i=1}^n
\varepsilon_i
\{[f_{\theta}(X_i)-\log\mu(Y_i;
f_{\theta})]\Delta_i \\
&& \qquad \quad - \ [f_{\theta^*}(X_i)-\log\mu(Y_i;
f_{\theta^*})]\Delta_i\}\Bigg|\Bigg)\\
&\leq& 2E\Bigg(\Bigg|\frac{1}{n}\sum_{i=1}^n \varepsilon_i
\{f_{\theta}(X_i)-f_{\theta^*}(X_i)\}\Delta_i\Bigg|\Bigg)\\
&& \qquad \quad + \ 2E\Bigg(\Bigg|\frac{1}{n}\sum_{i=1}^n \varepsilon_i
\{\log\mu(Y_i; f_{\theta})-\log\mu(Y_i;
f_{\theta^*})\}\Delta_i\Bigg|\Bigg)\\
&=& A+B.
\end{eqnarray*}

For $A$ we have
\begin{equation*}
A \leq 2 \Bigg(\sum_{k=1}^m \sigma_k
|\theta_k-\theta_{k}^*|\Bigg)E\Bigg(\max_{1\leq k \leq m}
\Bigg|\frac{1}{n}\sum_{i=1}^n\varepsilon_i
\Delta_i\psi_k(X_i)/\sigma_k\Bigg|\Bigg).
\end{equation*}
Applying Lemma A.1 in \cite{geer2008} with $\eta_n=K_m$ and
$\tau_{n}^2=K_m^2$, we obtain
\begin{equation*}
E\Bigg(\max_{1\leq k \leq m}\Bigg|\frac{1}{n}\sum_{i=1}^n
\varepsilon_i \Delta_i \frac{\psi_k(X_i)}{\sigma_k}\Bigg|\Bigg) \leq {a_n}.
\end{equation*}
Thus we have
\begin{equation} \label{e1}
A\leq 2{a}_n  M .
\end{equation}

For $B$, instead of using the contraction theorem that requires Lipschitz, we use the mean value theorem in the following:

\begin{eqnarray*}
&&\hspace{-0.2in} \Bigg|\frac{1}{n}\sum_{i=1}^n \varepsilon_i\{\log\mu(Y_i;
f_{\theta})-\log\mu(Y_i; f_{\theta^*})\}\Delta_i\Bigg|\\
&& \ = \Bigg|\frac{1}{n}\sum_{i=1}^n \varepsilon_i \Delta_i \sum_{k=1}^m
\frac{1}{\mu(Y_i;f_{\theta^{**}})}\int_{Y_i}^{\infty} \!\!\! \int_{\mathscr X}(\theta_k-\theta_{k}^*)\psi_{k}(x)e^{f_{\theta^{**}}(x)}dP_{X,Y}(x,y) \Bigg|\\
&& = \ \Bigg|\sum_{k=1}^m\sigma_k(\theta_k-\theta_{k}^*)\frac{1}{n}\sum_{i=1}^n
\frac{\varepsilon_i
\Delta_i}{\mu(Y_i;f_{\theta^{**}})\sigma_k} \int_{Y_i}^{\infty} \!\!\! \int_{\mathscr X}\psi_{k}(x)e^{f_{\theta^{**}}(x)}dP_{X,Y}(x,y) \Bigg|\\
&& \leq
\ \Bigg|\sum_{k=1}^m\sigma_k(\theta_k-\theta_{k}^*)\Bigg|\max_{1\leq
k \leq m}\Bigg|\frac{1}{n}\sum_{i=1}^n\varepsilon_i \Delta_i
F_{\theta^{**}}(k, Y_i)\Bigg|\\
&&\leq \ M\max_{1 \leq k \leq
m}\Bigg|\frac{1}{n}\sum_{i=1}^n\varepsilon_i \Delta_i
F_{\theta^{**}}(k, Y_i)\Bigg|,
\end{eqnarray*}
where $\theta^{**}$ is between $\theta$ and $\theta^*$, and
\begin{eqnarray}
F_{\theta^{**}}(k,t) &=& \frac{E \left[1(Y\geq
t)\psi_k(X)e^{f_{\theta^{**}}(X)} \right]}{\mu(t;f_{\theta^{**}})\sigma_k} \nonumber \\
&\leq&
\frac{(\|\psi_k\|_{\infty}/\sigma_k) E\left[1(Y\geq
t)e^{f_{\theta^{**}}(X)}\right]}{\mu(t;f_{\theta^{**}})} \ \leq  \ K_m.
\label{ftheta}
\end{eqnarray}
Since for all $i$,
\begin{eqnarray*}
&& E[\varepsilon_i\Delta_i F_{\theta^{**}}(k,Y_i)]=0, ~~~
\|\varepsilon_i\Delta_i F_{\theta^{**}}(k,Y_i)\|_{\infty}\leq
K_m, \ \mbox{and} \\
&& \frac{1}{n}\sum_{i=1}^n E[\varepsilon_i\Delta_i
F_{\theta^{**}}(k,Y_i)]^2\leq \frac{1}{n}\sum_{i=1}^n
E[F_{\theta^{**}}(k,Y_i)]^2\leq EK_{m}^2=K_{m}^2,
\end{eqnarray*}
following Lemma A.1 in \cite{geer2008}, we obtain
\begin{equation}
B\leq 2{a_n}M . \label{e2}
\end{equation}
Combining ({\ref{e1}}) and ({\ref{e2}}), the upper bound for $E Z_\theta(M)$
is achieved.
\end{proof}

We now can bound $Z_\theta (M)$ using the Bousquet's concentration theorem provided in \cite{geer2008} as Theorem A.1.

\begin{cor} \label{Cor 3.1}
Under Assumptions A, D and E, for all $M>0$, $r_1>0$ and all $\theta$
satisfying $I(\theta-\theta^*)\leq M$, it holds that
\begin{equation*}
P\left(Z_{\theta}(M) \geq \bar{\lambda}_{n,0}^A M \right) \leq
\exp\left(-n\bar{a}_{n}^2 r_{1}^2\right),
\end{equation*}
where
\begin{equation*}
\bar{\lambda}_{n,0}^A:=\bar{\lambda}_{n,0}^A(r_1):= \bar{a}_n
\left(1+2r_1
\sqrt{2\left(K_m^2+\bar{a}_nK_m\right)}+\frac{4r_{1}^2\bar{a}_n
K_m}{3}\right)
\end{equation*}
\end{cor}

\begin{proof}
Using the triangular inequality and the  mean value theorem, we obtain
\begin{eqnarray*}
|\gamma_{f_\theta}-\gamma_{f_{\theta^*}}|& \!\!\! \leq& \!\!\!
|f_{\theta}(X)-f_{\theta^*}(X)|\Delta+|\log\mu(Y;
f_{\theta})-\log\mu(Y;
f_{\theta^*})|\Delta\\
&\leq & \!\!\! \left|\sum_{k=1}^m \sigma_k |\theta_k-\theta_k^*|
\frac{\psi_k(X)}{\sigma_k}\right|+\left|\log\mu(Y;f_{\theta})-\log\mu(Y;f_{\theta^*})\right|\\
&\leq & \!\!\! MK_m +\sum_{k=1}^m\sigma_k |\theta_k-\theta_k^*|\cdot
\max_{1\leq k\leq m} \left|F_{\theta^{**}}(k,Y)\right|\\
&\leq & \!\!\! 2MK_m,
\end{eqnarray*}
where $\theta^{**}$ is between $\theta$ and $\theta^*$, $F_{\theta^{**}}(k,Y)$ is defined in (\ref{ftheta}).
So we have
\begin{equation*}
\|\gamma_{f_\theta}-\gamma_{f_{\theta^*}}\|_{\infty}\leq 2MK_m,
\end{equation*}
and
\begin{equation*}
P(\gamma_{f_\theta}-\gamma_{f_{\theta^*}})^2 \leq 4M^2K_m^2.
\end{equation*}
Therefore, in view of Bousquet's concentration theorem and Lemma \ref{Lemma 3.1}, for all
$M>0$ and $r_1>0$,
\begin{eqnarray*}
&&P\left(Z_{\theta}(M) \geq \bar{a}_n M\left(1+2r_1
\sqrt{2\left(K_m^2+\bar{a}_nK_m\right)}+\frac{4r_{1}^2\bar{a}_n
K_m}{3}\right)\right)\\
&& \qquad \qquad  \leq \  \exp\left(-n\bar{a}_{n}^2 r_{1}^2\right).
\end{eqnarray*}
\end{proof}

Now for any $\theta$
satisfying $I(\theta-\theta^*)\leq M$, we bound
$$R_{\theta}(M):=\left|\left[l_n(\theta)-\tilde{l}_n(\theta)\right]-\left[l_n(\theta^*)-\tilde{l}_n(\theta^*)\right]\right|,$$ which is equal to
\begin{eqnarray*}
&&\hspace{-0.2in}\frac{1}{n}\sum_{i=1}^n\Bigg|\Bigg[\log\frac{1}{n}\sum_{j=1}^n\frac{1(Y_j\geq
Y_i)e^{f_{\theta}(X_j)}}{\mu(Y_i;f_{\theta})}-\log\frac{1}{n}\sum_{j=1}^n\frac{1(Y_j\geq
Y_i)e^{f_{\theta^*}(X_j)}}{\mu(Y_i;f_{\theta^*})}\Bigg]\Delta_i\Bigg| \\
&&\leq \sup_{0\leq t\leq \tau}
\Bigg|\log\frac{1}{n}\sum_{j=1}^n\frac{1(Y_j\geq
t)e^{f_{\theta}(X_j)}}{\mu(t;f_{\theta})}-\log\frac{1}{n}\sum_{j=1}^n\frac{1(Y_j\geq
t)e^{f_{\theta^*}(X_j)}}{\mu(t;f_{\theta^*})}\Bigg|.
\end{eqnarray*}

By the mean value theorem, we have

\begin{eqnarray}
R_\theta(M) & \leq  &  \sup_{0\leq t\leq \tau}\Bigg|\sum_{k=1}^m(\theta_k-\theta_k^*)\Bigg\{\frac{\sum_{j=1}^n 1(Y_j\geq
t)e^{f_{\theta^{**}}(X_j)}}{
{\mu(t;f_{\theta^{**}})}}\Bigg\}^{-1} \nonumber \\
&&
\quad \Bigg\{\frac{\sum_{j=1}^n1(Y_j\geq
t)\psi_k(X_j)e^{f_{\theta^{**}}(X_j)}}{\mu(t;f_{\theta^{**}})} \nonumber \\
&&\qquad  - \ \frac{\sum_{j=1}^n 1(Y_j\geq t)e^{f_{\theta^{**}}(X_j)}E\left[1(Y\geq
t)\psi_k(X)e^{f_{\theta^{**}}(X)}\right]}{\mu(t;f_{\theta^{**}})^2}\Bigg\}\Bigg| \nonumber \\
& =&\sup_{0\leq t\leq
\tau}\Bigg|\sum_{k=1}^m\sigma_k(\theta_k-\theta_k^*) \Bigg\{\frac{\sum_{j=1}^n1(Y_j\geq
t)\{\psi_k(X_j)/\sigma_k\}e^{f_{\theta^{**}}(X_j)}}{\sum_{j=1}^n1(Y_j\geq
t)e^{f_{\theta^{**}}(X_j)}} \nonumber \\
&& \qquad - \ \frac{E\left[1(Y\geq
t)\{\psi_k(X)/\sigma_k\}e^{f_{\theta^{**}}(X)}\right]}{E\left[1(Y\geq
t)e^{f_{\theta^{**}}(X)}\right]}\Bigg\}\Bigg| \nonumber \\
&\leq& M  \sup_{0\leq t\leq \tau} \Bigg[\frac{1}{n}\sum_{i=1}^n 1(Y_i\geq t)
 e^{f_{\theta^{**}}(X_i)}\Bigg]^{-1}  \label{R-Bound} \\
&& \quad \ \sup_{0\leq t\leq \tau} \Bigg\{\max_{1\leq k\leq
m} \Bigg|\frac{1}{n}\sum_{i=1}^n 1(Y_i\geq t)
\{\psi_k(X_i)/\sigma_k\}
e^{f_{\theta^{**}}(X_i)} \nonumber \\
&& \qquad \qquad - \ E\left[1(Y\geq t)\{\psi_k(X)/\sigma_k\}
e^{f_{\theta^{**}}(X)}\right]\Bigg| \nonumber  \\
&& \qquad  + \ K_m \Bigg|\frac{1}{n} \sum_{i=1}^n 1(Y_i\geq
t)
 e^{f_{\theta^{**}}(X_i)}-E\left[1(Y\geq
t) e^{f_{\theta^{**}}(X)}\right]\Bigg|\Bigg\}, \nonumber
\end{eqnarray}
where $\theta^{**}$ is between $\theta$ and $\theta^*$,  and by (\ref{bup}) we have
\begin{eqnarray} \label{R-Bound-Common-Factor}
&& \sup_{0\leq t\leq \tau} \Bigg[\frac{1}{n}\sum_{i=1}^n 1(Y_i\geq t)
 e^{f_{\theta^{**}}(X_i)}\Bigg]^{-1} \le U_m  \Bigg[\frac{1}{n}\sum_{i=1}^n 1(Y_i\geq \tau) \Bigg]^{-1}.
\end{eqnarray}

\begin{lem} \label{Lemma 3.2}
Under Assumption E, we have
\begin{equation*}
P\left(\frac{1}{n}\sum_{i=1}^n 1(Y_i\geq \tau)
 \leq \frac{\pi}{2}\right)\leq 2 e^{-n\pi^2/2}. \label{b0}
\end{equation*}
\end{lem}

\begin{proof}
This is obtained directly from \cite{massart1990} by taking
$r=\pi\sqrt{n}/2$ in the following:
\begin{eqnarray*}
P\left(\frac{1}{n}\sum_{i=1}^n 1(Y_i\geq \tau) \leq \frac{\pi}{2} \right) &\le & P\left(\sup_{0\leq t\leq \tau} \sqrt{n}\left|\frac{1}{n}\sum_{i=1}^n1(Y_i\geq
t)-\pi\right|\geq r\right) \\
&\leq& 2e^{-2r^2}.
\end{eqnarray*}


\end{proof}

\begin{lem} \label{Lemma 3.3}
Under Assumptions A, D and E, for all $\theta$ we have
\begin{eqnarray}
&& P\left(\sup_{0\leq t\leq\tau}\left|\frac{1}{n}\sum_{i=1}^n1(Y_i\geq
t)e^{f_{\theta}(X_i)}-\mu(t;f_{\theta})\right|\geq
U_m\bar{a}_nr_1\right) \label{b1} \\
&& \qquad \qquad \leq
\ \frac{1}{5}W^2e^{-n\bar{a}_n^2r_1^2}, \nonumber
\end{eqnarray}
where $W$ is a constant that only depends on $K=\sqrt{2}$.

\end{lem}

\begin{proof}

For a class of functions indexed by $t$, $\mathscr{F}=\{1(y\geq
t)e^{f_{\theta}(x)}/U_m:t\in [0,\tau], y\in {\bf{R}}, e^{f_{\theta}(x)}\leq U_m\}$, we calculate its
bracketing number. For any $\epsilon > 0$, let $t_i$ be the $i$-th $\lceil{1}/{\varepsilon}\rceil$ quantile of $Y$, i.e.,
\begin{equation*}
P(Y\le t_i) = i\varepsilon, \quad i=1,
\cdots,\lceil{1}/{\varepsilon}\rceil-1,
\end{equation*}
where $\lceil x\rceil$ is the smallest integer that is greater than or
equal to $x$. Furthermore, denote $t_0=0$ and
$t_{\lceil{1}/{\varepsilon}\rceil}=+\infty$. For $i=1,\cdots,\lceil{1}/{\varepsilon}\rceil$,
define brackets $[L_i,U_i]$ with
\begin{eqnarray*}
L_i(x,y)=1(y\geq t_i)e^{f_{\theta}(x)}/U_m, \ U_i(x,y)=1(y > t_{i-1})e^{f_{\theta}(x)}/U_m
\end{eqnarray*}
such that $L_i(x,y) \le 1(y\geq t)e^{f_{\theta}(x)}/U_m  \le U_i(x,y)$ when $t_{i-1} < t \le t_i$. Since
\begin{eqnarray*}
\left\{E[U_i-L_i]^2\right\}^{1/2}&\leq&\left\{E\left[\frac{e^{f_{\theta}(X)}}{U_m}\{1(Y\geq
t_i)-1(Y > t_{i-1})\}\right]^2\right\}^{1/2} \\
&\leq & \left\{P(t_{i-1} < Y\leq
t_i)\right\}^{1/2}=\sqrt{\varepsilon},
\end{eqnarray*}
we have $N_{[\, ]}(\sqrt{\varepsilon},\mathscr{F},L_2)\leq
2/{\varepsilon}$, which yields
\begin{equation*}
N_{[\, ]}(\varepsilon,\mathscr{F},L_2)\leq
\frac{2}{\varepsilon^2} =
\left(\frac{K}{\varepsilon}\right)^2,
\end{equation*}
where $K=\sqrt{2}$. Thus, from Theorem 2.14.9 in \cite{vaart1996},
we have for any $r>0$,
\begin{eqnarray*}
P\Bigg(\sqrt{n}\sup_{0\leq
t\leq\tau}\Bigg|\frac{1}{n}\sum_{i=1}^n \frac{1(Y_i\geq
t)e^{f_{\theta}(X_i)}}{U_m}- \frac{\mu(t;f_{\theta})}{U_m} \Bigg|\geq
r\Bigg)&\leq&  \frac{1}{2} W^2r^2 e^{-2r^2} \\
& \le & \frac{1}{5}W^2e^{-r^2},
\end{eqnarray*}
where $W$ is a constant that only depends on $K$. Note that $r^2e^{-r^2}$ is bounded by $e^{-1}$.
 Let $r=\sqrt{n}\bar{a}_nr_1$, we obtain ($\ref{b1}$).

\end{proof}

\begin{lem} \label{Lemma 3.4}
Under Assumptions A, D and E, for all $\theta$ we have
\begin{eqnarray}
&&P\Bigg(\sup_{0\leq t\leq\tau}\max_{0\leq k\leq
m}\Bigg|\frac{1}{n}\sum_{i=1}^n1(Y_i\geq
t)\frac{\psi_k(X_i)}{\sigma_k}e^{f_{\theta}(X_i)}\notag\\
&&\qquad \qquad - \ E\left[1(Y\geq
t)\frac{\psi_k(X)}{\sigma_k}e^{f_{\theta}(X)}\right]\Bigg|\geq
 K_mU_m\Bigg[\bar{a}_nr_1+\sqrt{\frac{\log(2m)}{n}}\Bigg]\Bigg)\notag\\
 &&\qquad \leq \ \frac{1}{10}W^2e^{-n\bar{a}_n^2r_1^2}.\label{a1}
\end{eqnarray}

\end{lem}
\begin{proof}

Consider the classes of functions indexed by $t$,
\begin{eqnarray*}
\mathscr{G}^k &=&\big\{1(y\geq
t)e^{f_{\theta}(x)}{\psi_k(x)}/(\sigma_k K_m U_m):t\in [0,\tau],y\in
{\bf{R}}, \\
&& \qquad \big|e^{f_{\theta}(x)}{\psi_k(x)}/{\sigma_k}\big|\leq
K_mU_m\big\}, \quad k=1, \dots, m.
\end{eqnarray*}
Using the same argument in the proof of Lemma \ref{Lemma 3.3}, we have
\begin{equation*}
N_{[\, ]}(\varepsilon,\mathscr{G}^k,L_2)\leq
\left(\frac{K}{\varepsilon}\right)^2,
\end{equation*}
where $K = \sqrt{2}$, and then for any
$r>0$,
\begin{eqnarray*}
&& P\Bigg(\sqrt{n}\sup_{0\leq
t\leq\tau}\Bigg|\frac{1}{n}\sum_{i=1}^n \frac{1(Y_i\geq
t)e^{f_{\theta}(X_i)}\psi_k(X_i)}{\sigma_k K_m U_m} \\
&& \qquad \qquad \qquad - \ E\left[\frac{1(Y\geq
t)e^{f_{\theta}(X)}\psi_k(X)}{\sigma_k K_m U_m}\right] \Bigg|\geq
r\Bigg)
 \le  \frac{1}{5}W^2e^{-r^2}.
\end{eqnarray*}

Thus we have
\begin{eqnarray*}
&&P\bigg(\sqrt{n}\sup_{0\leq t\leq\tau}\max_{0\leq k\leq
m}\bigg|\frac{1}{n}\sum_{i=1}^n1(Y_i\geq
t)e^{f_{\theta}(X_i)}{\psi_k(X_i)}/\left({\sigma_k}U_mK_m\right)\\
&&\qquad \qquad - \ E\left[1(Y\geq
t)e^{f_{\theta}(X)}{\psi_k(X)}/\left({\sigma_k}U_mK_m\right)\right]\bigg|\geq
r\bigg)\\
&&\qquad \leq \ P\bigg(\bigcup_{k=1}^m\sqrt{n}\sup_{0\leq
t\leq\tau}\bigg|\frac{1}{n}\sum_{i=1}^n1(Y_i\geq
t)e^{f_{\theta}(X_i)}{\psi_k(X_i)}/\left({\sigma_k}U_mK_m\right)\\
&&\qquad \qquad - \ E\left[1(Y\geq
t)e^{f_{\theta}(X)}{\psi_k(X)}/\left({\sigma_k}U_mK_m\right)\right]\bigg|\geq
r\bigg)\\
&&\qquad \leq \ mP\bigg(\sqrt{n}\sup_{0\leq
t\leq\tau}\bigg|\frac{1}{n}\sum_{i=1}^n1(Y_i\geq
t)e^{f_{\theta}(X_i)}{\psi_k(X_i)}/\left({\sigma_k}U_mK_m\right)\\
&&\qquad \qquad - \ E\left[1(Y\geq
t)e^{f_{\theta}(X)}{\psi_k(X)}/\left({\sigma_k}U_mK_m\right)\right]\bigg|\geq
r\bigg)\\
&&\qquad \leq \ \frac{m}{5}W^2e^{-r^2}  =  \frac{1}{10}W^2e^{\log(2m)-r^2}.
\end{eqnarray*}
Let $\log(2m)-r^2=-n\bar{a}_n^2r_1^2$, i.e.
$r=\sqrt{n\bar{a}_n^2r_1^2+\log(2m)}$.
Since
$$
\sqrt{\bar{a}_n^2r_1^2+\frac{\log(2m)}{n}}\leq\bar{a}_nr_1+\sqrt{\frac{\log(2m)}{n}},
$$
we obtain (\ref{a1}).

\end{proof}

\begin{cor} \label{Cor 3.2}
Under Assumptions A, D and E, for all $M>0$ and all $\theta$ that
satisfies $I(\theta-\theta^*)\leq M$, we have
\begin{equation}
P\left(R_{\theta}(M) \geq \bar{\lambda}_{n,0}^BM \right) \leq
2\exp\left(-n\pi^2/2\right)+\frac{3}{10}W^2\exp\left(-n\bar{a}_n^2r_1^2\right),\label{dm}
\end{equation}
\end{cor}
where
$$\bar{\lambda}_{n,0}^B=\frac{2K_mU_m^2}{\pi}\left(2\bar{a}_nr_1+\sqrt{\frac{\log(2m)}{n}}\right) . $$

\begin{proof}
From inequalities (\ref{R-Bound}) and (\ref{R-Bound-Common-Factor})
we have
$$
P\left(R_{\theta}(M)\leq \bar{\lambda}_{n,0}^B\cdot M\right) \ge P\left(E_1^c \cap E_2^c \cap E_3^c\right),
$$
where the events $E_1$, $E_2$ and $E_3$ are defined in the following:
\begin{eqnarray*}
E_1 &=& \Bigg\{\frac{1}{n}\sum_{i=1}^n1(Y_i\geq \tau)\leq
\pi/2\Bigg\},\\
E_2 &=& \Bigg\{\sup_{0\leq
t\leq\tau}\Bigg|\frac{1}{n}\sum_{i=1}^n1(Y_i\geq
t)e^{f_{\theta^{**}}(X_i)}-\mu(t;f_{\theta^{**}})\Bigg|\geq
U_m\bar{a}_nr_1\Bigg\}, \\
E_3 &=& \Bigg\{\max_{0\leq k\leq m}\sup_{0\leq
t\leq\tau}\Bigg|\frac{1}{n}\sum_{i=1}^n1(Y_i\geq
t)\frac{\psi_k(X_i)}{\sigma_k}e^{f_{\theta^{**}}(X_i)}\\
&&\quad - \ E\left[1(Y\geq
t)\frac{\psi_k(X)}{\sigma_k}e^{f_{\theta^{**}}(X)}\right]\Bigg|\geq
K_mU_m(\bar{a}_nr_1+\sqrt{\frac{\log(2m)}{n}})\Bigg\}.
\end{eqnarray*}
Thus
$$
P\left(R_{\theta}(M)\geq \bar{\lambda}_{n,0}^B\cdot M\right) \le P\left(E_1^c\right) +P\left( E_2^c \right) + P\left( E_3^c\right),
$$
and the result follows from Lemmas \ref{Lemma 3.2}, \ref{Lemma 3.3} and \ref{Lemma 3.4}.

\end{proof}

We now show oracle bounds for the lasso estimator $\hat\theta_n$ under Assumptions A-E following \cite{geer2008}, but using pointwise arguments.
Let
\begin{eqnarray} \label{lambda-n-0}
\bar{\lambda}_{n,0}=\bar{\lambda}_{n,0}^A+\bar{\lambda}_{n,0}^B.
\end{eqnarray}
Take $b>0$, $d>1$, and
\begin{equation*}
d_b:=d\left(\frac{b+d}{(d-1)b}\vee 1\right).
\end{equation*}
Let $D_\theta := D(\{k: \theta_k \ne 0, k= 1, \dots,m\})$ be the number of nonzero $\theta_k$'s, where $D(\cdot)$ is given in Assumption C. Define
\begin{eqnarray*}
&&(A1)~~\lambda_n:=(1+b)\bar{\lambda}_{n,0},\\
&&(A2)~~{\cal V}_{\theta}:=2\delta
H\left(\frac{2\lambda_n\sqrt{D_{\theta}}}{\delta}\right), \ where \ 0<\delta <1,\\
&&(A3)~~\theta_n^*:={\rm argmin}_{\theta\in\Theta}\{\mathcal{E}(f_{\theta})+{\cal V}_{\theta}\},\\
&&(A4)~~\epsilon_n^*:=(1+\delta)\mathcal{E}(f_{\theta_n^*})+{\cal V}_{\theta_n^*},\\
&&(A5)~~\zeta_n^*:=\frac{\epsilon_n^*}{\bar{\lambda}_{n,0}},\\
&&(A6)~~\theta(\epsilon_n^*):={\rm argmin}_{\theta \in \Theta, I(\theta-\theta_n^*)\leq
d_b\zeta_n^*/b}\{\delta\mathcal{E}(f_{\theta})-2\lambda_nI_1(\theta-\theta_n^*|\theta_n^*)\}.
\end{eqnarray*}

\smallskip

We also impose the following conditions:

\smallskip

{\sc Condition} I$(b,\delta)$. \ $\|f_{\theta_n^*}-\bar{f}\|_{\infty} \leq
\eta$.

\smallskip

{\sc Condition} II$(b,\delta,d)$. \
$\|f_{\theta(\epsilon_n^*)}-\bar{f}\|_{\infty} \leq \eta$.

\smallskip

In both conditions, $\eta$ is given in Assumption B.

\begin{lem} \label{Lemma 3.5}
Suppose Conditions I$(b,\delta)$ and II$(b,\delta,d)$ are met. For
all $\theta\in\Theta$ with $I(\theta-\theta_n^*)\leq
d_b\zeta_n^*/b$, it holds that
\begin{equation*}
2\lambda_nI_1(\theta-\theta_n^*)\leq\delta\mathcal{E}(f_\theta)+
\epsilon_n^*-\mathcal{E}(f_{\theta_n^*}).
\end{equation*}
\end{lem}
\begin{proof}
The proof is exactly the same as that of Lemma A.4 in \cite{geer2008}, with $\lambda_n$ defined in (\ref{lambda-n-0}).
\end{proof}

\begin{lem} \label{Lemma 3.6}
Suppose Conditions I$(b,\delta)$ and II$(b,\delta,d)$ are met.
Consider any random $\tilde{\theta}\in \Theta$ with
$l_n(\tilde{\theta})+\lambda_n I(\tilde{\theta})\leq
l_n(\theta_n^*)+\lambda_n I(\theta_n^*)$. Let $1<d_0\leq d_b$. It
holds that
\begin{eqnarray*}
P\left(I(\tilde{\theta}-\theta_n^*)\leq d_0
\frac{\zeta_n^*}{b}\right)&\!\!\!\leq\!\!\!&
P\left(I(\tilde{\theta}-\theta_n^*)\leq
\left(\frac{d_0+b}{1+b}\right)\frac{\zeta_n^*}{b}\right)\\
&& + \ \left(1+\frac{3}{10}W^2\right)\exp\left(-n\bar{a}_n^2r_1^2\right)+2\exp\left(-n\pi^2/2\right).
\end{eqnarray*}
\end{lem}

\begin{proof}
The idea is similar to the proof of Lemma A.5 in \cite{geer2008}.  Let $
\tilde{\mathcal{E}}=\mathcal{E}(f_{\tilde{\theta}})$ and
$\mathcal{E}^*=\mathcal{E}(f_{\theta_n^*})$. We will use short
notation: $I_1(\theta)=I_1(\theta|\theta_n^*)$ and
$I_2(\theta)=I_2(\theta|\theta_n^*)$. Since  $l_n(\tilde{\theta})+\lambda_n I(\tilde{\theta})\leq
l_n(\theta_n^*)+\lambda_n I(\theta_n^*)$, on the set where
$I(\tilde{\theta}-\theta_n^*)\leq d_0\zeta_n^*/b$ and
$Z_{\tilde{\theta}}(d_0\zeta_n^*/b)\leq d_0\zeta_n^*/b\cdot
\bar{\lambda}_{n,0}^A$, we have
\begin{eqnarray}
R_{\tilde{\theta}}(d_0\zeta_n^*/b) & \geq
&[l_n(\theta_n^*)+\lambda_n
I(\theta_n^*)]-[l_n(\tilde{\theta})+\lambda_n
I(\tilde{\theta})]-\lambda_n I(\theta_n^*)+\lambda_n
I(\tilde{\theta})\notag\\
&& -[\tilde{l}_n(\theta_n^*)-\tilde{l}_n(\tilde{\theta})]\notag\\
& \geq & -\lambda_n I(\theta_n^*)+\lambda_n
I(\tilde{\theta})-[\tilde{l}_n(\theta_n^*)-\tilde{l}(\tilde{\theta})]\notag\\
& \geq & -\lambda_n I(\theta_n^*)+\lambda_n
I(\tilde{\theta})-[l(\theta_n^*)-l(\tilde{\theta})]-d_0\zeta_n^*/b\cdot
\bar{\lambda}_{n,0}^A\notag\\
& \geq & -\lambda_n I(\theta_n^*)+\lambda_n
I(\tilde{\theta})-\mathcal{E}^*+\tilde{\mathcal{E}}-d_0
\bar{\lambda}_{n,0}^{A}{\zeta_n^*}/{b}.\label{lema}
\end{eqnarray}

By (\ref{dm}) we know that $R_{\tilde{\theta}}(d_0\zeta_n^*/b)$ is
bounded by $d_0 \bar{\lambda}_{n,0}^B{\zeta_n^*}/{b}$ with
probability at least
$1-\frac{3}{10}W^2\exp\left(-n\bar{a}_n^2r_1^2\right)-2\exp\left(-n\pi^2/2\right)$, then we
have
\begin{equation*}
\tilde{\mathcal{E}}+\lambda_n I(\tilde{\theta}) \leq
\bar{\lambda}_{n,0}^B d_0 \zeta_n^*/b +\mathcal{E}^*+\lambda_n
I(\theta_n^*)+\bar{\lambda}_{n,0}^A d_0 \zeta_n^*/b.\label{lemb}
\end{equation*}
Since $I(\tilde{\theta})=I_1(\tilde{\theta})+I_2(\tilde{\theta})$ and
$I(\theta_n^*)=I_1(\theta_n^*)$, using the triangular inequality, we obtain
\begin{eqnarray}
&& \tilde{\mathcal{E}}+(1+b)\bar{\lambda}_{n,0} I_2({\tilde{\theta}}) \notag\\
&& \qquad \qquad \leq \
\bar{\lambda}_{n,0}d_0\zeta_n^*/b+\mathcal{E}^*+(1+b)\bar{\lambda}_{n,0}
I_1(\theta_n^*)-(1+b)\bar{\lambda}_{n,0} I_1(\tilde{\theta})\notag\\
&& \qquad \qquad  \leq \
\bar{\lambda}_{n,0}d_0\zeta_n^*/b+\mathcal{E}^*+(1+b)\bar{\lambda}_{n,0}
I_1(\tilde{\theta}-\theta_n^*).\label{lemc}
\end{eqnarray}
The remaining of the proof follows exactly the same as the corresponding part of the proof of Lemma A.5 in van de Geer (2008).
\end{proof}


\begin{cor} \label{Cor 3.3}
Suppose Conditions I$(b,\delta)$ and II$(b,\delta,d)$ are met.
Consider any random $\tilde{\theta}\in \Theta$ with
$l_n(\tilde{\theta})+\lambda_n I(\tilde{\theta})\leq
l_n(\theta_n^*)+\lambda_n I(\theta_n^*)$. Let $ 1 < d_0 \le d_b$. It holds
that
\begin{eqnarray*}
&& P\left(I(\tilde{\theta}-\theta_n^*)\leq d_0
\frac{\zeta_n^*}{b}\right) \\
&& \qquad \qquad \le P\left(I(\tilde{\theta}-\theta_n^*)\leq
\left[1+(d_0-1)(1+b)^{-N}\right]\frac{\zeta_n^*}{b}\right)\\
&&\qquad \qquad \qquad + \ N
\left\{\left(1+\frac{3}{10}W^2\right)\exp\left(-n\bar{a}_n^2r_1^2\right)+2\exp(-n\pi^2/2)\right\}.
\end{eqnarray*}
\end{cor}

\begin{proof}
Repeat Lemma \ref{Lemma 3.6} $N$ times.
\end{proof}

\begin{lem} \label{Lemma 3.7}
Suppose Conditions I$(b,\delta)$ and II$(b,\delta,d)$
are met. Define
\begin{equation*}
\tilde{\theta}_s=s\hat{\theta}_n+(1-s)\theta_n^*,
\end{equation*}
where
\begin{equation*}
s=\frac{d \zeta_n^*}{d \zeta_n^*+bI(\hat{\theta}_n-\theta_n^*)}.
\end{equation*}
Then for any integer $N$, with probability at least
$$1-N\left\{\left(1+\frac{3}{10}W^2\right)\exp\left(-n\bar{a}_n^2r_1^2\right)+2\exp\left(-n\pi^2/2\right)\right\},$$
we have
\begin{equation*}
I(\tilde{\theta}_s-\theta_n^*)\leq
\left(1+(d-1)(1+b)^{-N}\right)\frac{\zeta_n^*}{b}.
\end{equation*}
\end{lem}

\begin{proof}
Since the negative log partial likelihood $l_n(\theta)$ and the lasso
penalty are both convex with respect to $\theta$, applying Corollary \ref{Cor
3.3}, we obtain the above inequality.
\end{proof}

\begin{lem} \label{Lemma 3.8}
Suppose Conditions I$(b,\delta)$ and II$(b,\delta,d)$ are met. Let
$N_1\in {\bf N} :=\{1,2,\dots\}$ and $N_2\in {\bf N}\cup \{0\}$. Define
$\delta_1=(1+b)^{-N_1}$ and $\delta_2=(1+b)^{-N_2}$.
For any $n$, with probability at least
$$1-(N_1+N_2)\left\{\left(1+\frac{3}{10}W^2\right)\exp\left(-n\bar{a}_n^2r_1^2\right)+2\exp(-n\pi^2/2)\right\},
$$ we have
\begin{equation*}
I(\hat{\theta}_n-\theta_n^*)\leq d(\delta_1,
\delta_2)\frac{\zeta_n^*}{b},
\end{equation*}
where
\begin{equation*}
d(\delta_1,
\delta_2)=1+\frac{1+(d^2-1)\delta_1}{(d-1)(1-\delta_1)}\delta_2.
\end{equation*}
\end{lem}

\begin{proof}
The proof is exactly the same as that of Lemma A.7 in \cite{geer2008}, with a slightly different probability bound.
\end{proof}

We now provide the major theorem of the oracle inequalities for the
Cox model lasso estimator.

\begin{thm} \label{Thm 3.1}
Suppose Assumptions A-E and Conditions I$(b,\delta)$ and
II$(b,\delta,d)$ are met. Let
$$
\Delta(b,\delta,\delta_1,\delta_2):=d(\delta_1,\delta_2)\frac{1-\delta^2}{\delta
b}\vee 1.
$$
We have with probability at least
\begin{eqnarray*}
&& 1-\Bigg\{\log_{1+b}\frac{(1+b)^2\Delta(b,\delta,\delta_1,\delta_2)}{\delta_1\delta_2}\Bigg\}
\Bigg\{\left(1+\frac{3}{10}W^2\right)\exp\left(-n\bar{a}_n^2r_1^2\right) \\
&& \hspace{2.5in} + \ 2\exp\left(-n\pi^2/2\right)\Bigg\}
\end{eqnarray*}
that
\begin{equation*}
\mathcal{E}(f_{\hat{\theta}_n})\leq \frac{1}{1-\delta}\epsilon_n^*,
\end{equation*}
and moreover,
\begin{equation*}
I(\hat{\theta}_n-\theta_n^*)\leq d(\delta_1,
\delta_2)\frac{\zeta_n^*}{b}.
\end{equation*}
\end{thm}

\begin{proof} The proof follows the same ideas in the proof of Theorem A.4 in \cite{geer2008}, with exceptions of pointwise arguments and slightly different probability bounds. Since this is the major result, to be self-contained, we provide a detailed proof here despite the amount of overlaps.

Similar to \cite{geer2008}, we define
$\hat{\mathcal{E}}:=\mathcal{E}(f_{\hat{\theta}_n})$ and
$\mathcal{E}^*:=\mathcal{E}(f_{\theta_n^*})$; use the
notation $I_1(\theta):=I_1(\theta|\theta_n^*)$ and
$I_2(\theta):=I_2(\theta|\theta_n^*)$;
set $$c:=\frac{\delta b}{1-\delta^2};$$ and consider the cases (a)
$c<d(\delta_1, \delta_2)$ and (b) $c\geq d(\delta_1, \delta_2)$.

(a) Consider $c<d(\delta_1, \delta_2)$. Let $J$ be an
integer satisfying $(1+b)^{J-1}c\leq d(\delta_1, \delta_2)$ and
$(1+b)^{J}c> d(\delta_1, \delta_2)$. We consider the cases (a1)
$c\zeta_n^*/b<I(\hat{\theta}_n-\theta_n^*)\leq d(\delta_1,
\delta_2)\zeta_n^*/b$ and (a2)
$I(\hat{\theta}_n-\theta_n^*)\leq c\zeta_n^*/b$.

(a1) If $c\zeta_n^*/b<I(\hat{\theta_n}-\theta_n^*)\leq
d(\delta_1, \delta_2)\zeta_n^*/b$, then
\begin{equation*}
(1+b)^{j-1}c\frac{\zeta_n^*}{b}<I(\hat{\theta}_n-\theta_n^*)\leq
(1+b)^{j}c\frac{\zeta_n^*}{b}
\end{equation*}
for some $j\in \{1,\cdots,J\}$.
Let
\begin{equation*}
d_0=c(1+b)^{j-1}\leq d(\delta_1, \delta_2)\leq d_b.
\end{equation*}
From Corollary \ref{Cor 3.1}, with probability at least
$1-\exp\left(-n\bar{a}_n^2 r_1^2\right)$ we have
$Z_{\hat{\theta}_n}((1+b)d_0 \zeta_n^*/b)\leq
(1+b)d_0\bar{\lambda}_{n,0}^A \zeta_n^*/b$. Since
$
l_n(\hat{\theta}_n)+\lambda_nI(\hat{\theta}_n)\leq
l_n(\theta_n^*)+\lambda_nI(\theta_n^*),
$
from equation (\ref{lema}), we have
\begin{equation*}
\hat{\mathcal{E}}+\lambda_n I(\hat{\theta}_n)\leq
R_{\hat{\theta}_n}\left((1+b)d_0\frac{\zeta_n^*}{b}\right)+\mathcal{E}^*+\lambda_n
I(\theta_n^*)+(1+b)\bar{\lambda}_{n,0}^A d_0\frac{\zeta_n^*}{b}.
\end{equation*}
By (\ref{dm}),
$R_{\hat{\theta}_n}((1+b)d_0\zeta_n^*b)$ is
bounded by $(1+b)\bar{\lambda}_{n,0}^Bd_0\zeta_n^*/b$ with
probability at least
$$1-\frac{3}{10}W^2 \exp\left(-n\bar{a}_n^2r_1^2\right)-2\exp\left(-n\pi^2/2\right),$$ then we
have
\begin{eqnarray*}
\hat{\mathcal{E}}+(1+b)\bar{\lambda}_{n,0}I(\hat{\theta}_n)& \leq&
(1+b)\bar{\lambda}_{n,0}^Bd_0\frac{\zeta_n^*}{b}+
\mathcal{E}^*+(1+b)\bar{\lambda}_{n,0}I(\theta_n^*)\\
&& \qquad + \ (1+b)\bar{\lambda}_{n,0}^Ad_0\frac{\zeta_n^*}{b}\\
& \leq&
(1+b)\bar{\lambda}_{n,0}I(\hat{\theta}_n-\theta_n^*)+\mathcal{E}^*+(1+b)
\bar{\lambda}_{n,0}I(\theta_n^*).
\end{eqnarray*}
Since $I(\hat{\theta}_n)=I_1(\hat{\theta}_n)+I_2(\hat{\theta}_n)$,
$I(\hat{\theta}_n-\theta_n^*)=I_1(\hat{\theta}_n-\theta_n^*)+I_2(\hat{\theta}_n)$,
and $I(\theta_n^*)=I_1(\theta_n^*)$, by triangular inequality we obtain
\begin{equation*}
\hat{\mathcal{E}}\leq
2(1+b)\bar{\lambda}_{n,0}I_1(\hat{\theta}_n-\theta_n^*)+\mathcal{E}^*.
\end{equation*}
From Lemma \ref{Lemma 3.5},
\begin{eqnarray*}
\hat{\mathcal{E}}&\leq&
\delta\hat{\mathcal{E}}+\epsilon_n^*-\mathcal{E}^*+\mathcal{E}^*=\delta\hat{\mathcal{E}}+\epsilon_n^*.
\end{eqnarray*}
Hence,
\begin{equation*}
\hat{\mathcal{E}}\leq \frac{1}{1-\delta}\epsilon_n^*.
\end{equation*}

(a2) If $I(\hat{\theta}_n-\theta_n^*)\leq c\zeta_n^*/b$, from
equation (\ref{lemc}) with $d_0=c$, with probability at least
$$1-\left\{\left(1+\frac{3}{10}W^2\right)\exp\left(-n\bar{a}_n^2r_1^2\right)
+2\exp(-n\pi^2/2)\right\} ,$$ we have
\begin{equation*}
\hat{\mathcal{E}}+(1+b)\bar{\lambda}_{n,0}I(\hat{\theta}_n)\leq
\frac{\delta}{1- \delta^2}\bar{\lambda}_{n,0}\zeta_n^*+
\mathcal{E}^*+(1+b)\bar{\lambda}_{n,0}I(\theta_n^*).
\end{equation*}
By triangular inequality, Lemma \ref{Lemma 3.5} and ($A4$),
\begin{eqnarray*}
\hat{\mathcal{E}}& \leq&
\frac{\delta}{1-\delta^2}\bar{\lambda}_{n,0}\zeta_n^*+\mathcal{E}^*
+(1+b)\bar{\lambda}_{n,0}I_1(\hat{\theta}_n-\theta_n^*)\\
&\leq&\frac{\delta}{1-\delta^2}\bar{\lambda}_{n,0}\frac{\epsilon_n^*}{\bar{\lambda}_{n,0}}
+\mathcal{E}^*+\frac{\delta}{2}\hat{\mathcal{E}}+\frac{1}{2}\epsilon_n^*-\frac{1}{2}\mathcal{E}^*\\
&=&\left(\frac{\delta}{1-\delta^2}+\frac{1}{2}\right)\epsilon_n^*+\frac{1}{2}\mathcal{E}^*+\frac{\delta}{2}\hat{\mathcal{E}}\\
&\leq&\left(\frac{\delta}{1-\delta^2}+\frac{1}{2}\right)\epsilon_n^*+\frac{1}{2(1+\delta)}\epsilon_n^*+\frac{\delta}{2}\hat{\mathcal{E}}.
\end{eqnarray*}
Hence,
\begin{equation*}
\hat{\mathcal{E}}\leq
\frac{2}{2-\delta}\left[\frac{\delta}{1-\delta^2}+\frac{1}{2}+\frac{1}{2(1+\delta)}\right]\epsilon_n^*
=\frac{1}{1-\delta}\epsilon_n^*.
\end{equation*}

Furthermore, by Lemma \ref{Lemma 3.8},  we have with probability at least
$$1-(N_1+N_2)\left\{\left(1+\frac{3}{10}W^2\right)\exp\left(-n\bar{a}_n^2r_1^2\right)
+2\exp\left(-n\pi^2/2\right)\right\} $$ that
\begin{equation*}
I(\hat{\theta}_n-\theta_n^*)\leq
d(\delta_1,\delta_2)\frac{\zeta_n^*}{b},
\end{equation*}
where
\begin{eqnarray*}
N_1=\log_{1+b}\left(\frac{1}{\delta_1}\right), \ N_2=\log_{1+b}\left(\frac{1}{\delta_2}\right).
\end{eqnarray*}

(b) Consider $c\geq d(\delta_1,\delta_2)$. On the set where
$I(\hat{\theta}_n-\theta_n^*)\leq d(\delta_1,\delta_2)\zeta_n^*/b$,
from equation (\ref{lemc}) we have with probability at least
$$1-\left\{\left(1+\frac{3}{10}W^2\right)\exp\left(-n\bar{a}_n^2r_1^2\right)
+2\exp\left(-n\pi^2/2\right)\right\} $$ that
\begin{eqnarray*}
\hat{\mathcal{E}}+(1+b)\bar{\lambda}_{n,0}I(\hat{\theta}_n)&\leq&
\bar{\lambda}_{n,0}d(\delta_1,\delta_2)\frac{\zeta_n^*}{b}+\mathcal{E}^*
+(1+b)\bar{\lambda}_{n,0}I(\theta_n^*)\\
& \leq&
\frac{\delta}{1-\delta^2}\bar{\lambda}_{n,0}\zeta_n^*+\mathcal{E}^*+(1+b)\bar{\lambda}_{n,0}I(\theta_n^*),
\end{eqnarray*}
which is the same as (a2) and leads to the same result.

\medskip

To summarize, let
\begin{eqnarray*}
A=\left\{\hat{\mathcal{E}}\leq
\frac{1}{1-\delta}\epsilon_n^*\right\}, \quad
B=\left\{I(\hat{\theta}_n-\theta_n^*)\leq
d(\delta_1,\delta_2)\frac{\zeta_n^*}{b}\right\}.
\end{eqnarray*}
Note that
\begin{equation*}
J+1\leq \log_{1+b}\left(\frac{(1+b)^2d(\delta_1,\delta_2)}{c}\right).
\end{equation*}
Under case (a), we have
\begin{eqnarray*}
&&\hspace{-0.2in} P\left(A\cap B\right) \\
&& \ \ = \ P({\rm a1})-P(A^c\cap {\rm a1})+P({\rm a2})-P(A^c\cap {\rm a2})\\
&& \ \ \geq \
P({\rm a1})-J\left\{\left(1+\frac{3}{10}W^2\right)\exp\left(-n\bar{a}_n^2r_1^2\right)+2\exp(-n\pi^2/2)\right\}\\
&& \qquad \qquad  + \ P({\rm a2})-\left\{\left(1+\frac{3}{10}W^2\right)\exp(-n\bar{a}_n^2r_1^2)+2\exp\left(-n\pi^2/2\right)\right\}
\\
&& \ \ = \ P(B)-(J+1)\left\{\left(1+\frac{3}{10}W^2\right)\exp\left(-n\bar{a}_n^2r_1^2\right)+2\exp\left(-n\pi^2/2\right)\right\}
\\
&& \ \ \geq \ 1-(N_1+N_2+J+1)\Bigg\{\left(1+\frac{3}{10}W^2\right)\exp\left(-n\bar{a}_n^2r_1^2\right) \\
&& \ \ \qquad + \ 2\exp\left(-n\pi^2/2\right)\Bigg\}
\\
&& \ \ \geq \
1-\log_{1+b}\left\{\frac{(1+b)^2}{\delta_1\delta_2}\cdot\frac{d(\delta_1,\delta_2)(1-\delta^2)}
{\delta b}\right\}\\
&&\ \ \qquad \bigg\{\bigg(1+\frac{3}{10}W^2\bigg)
\exp\left(-n\bar{a}_n^2r_1^2\right)+2\exp(-n\pi^2/2)\bigg\} .
\end{eqnarray*}
Under case (b),
\begin{eqnarray*}
&&P\left(A\cap B\right) \\
&& \qquad = \ P(B)-P(A^c\cap B)\\
&&\qquad \geq \
P(B)-\left\{\left(1+\frac{3}{10}W^2\right)\exp\left(-n\bar{a}_n^2r_1^2\right)+2\exp\left(-n\pi^2/2\right)\right\}
\\
&&\qquad \geq \
1-(N_1+N_2+2)\Bigg\{\left(1+\frac{3}{10}W^2\right)\exp\left(-n\bar{a}_n^2r_1^2\right) \\
&& \qquad \qquad + \ 2\exp\left(-n\pi^2/2\right)\Bigg\}
\\
&&\qquad = \ 1-\log_{1+b}\left\{\frac{(1+b)^2}{\delta_1\delta_2}\right\} \\
&& \qquad \qquad \left\{\left(1+\frac{3}{10}W^2\right)\exp\left(-n\bar{a}_n^2r_1^2\right)+2\exp(-n\pi^2/2)\right\} .
\end{eqnarray*}
We thus obtain the desired result.
\end{proof}

\subsection{Random normalization weights in the penalty}

The case with random weights can be argued in the exactly the same way as that in \cite{geer2008}, for which the same tail probability given in Lemma A.9 of \cite{geer2008} is added to the probability bound in Theorem \ref{Thm 3.1} under the same set of conditions for Theorem A.5 in \cite{geer2008}. Thus details are omitted.

\end{document}